\documentclass{amsart} 
\usepackage{amssymb}
\usepackage{amsmath}
\usepackage{amsfonts}

\begin{document}
\newtheorem{theo}{Theorem}[section]
\newtheorem{atheo}{Theorem*}
\newtheorem{prop}[theo]{Proposition}
\newtheorem{aprop}[atheo]{Proposition*}
\newtheorem{lemma}[theo]{Lemma}
\newtheorem{alemma}[atheo]{Lemma*}
\newtheorem{exam}[theo]{Example}
\newtheorem{coro}[theo]{Corollary}
\theoremstyle{definition}
\newtheorem{defi}[theo]{Definition}
\newtheorem{rem}[theo]{Remark}


\newcommand{\Bb}{{\bf B}}
\newcommand{\Cb}{{\mathbb C}}
\newcommand{\Nb}{{\mathbb N}}
\newcommand{\Qb}{{\mathbb Q}}
\newcommand{\Rb}{{\mathbb R}}
\newcommand{\Zb}{{\mathbb Z}}
\newcommand{\Ac}{{\mathcal A}}
\newcommand{\Bc}{{\mathcal B}}
\newcommand{\Cc}{{\mathcal C}}
\newcommand{\Dc}{{\mathcal D}}
\newcommand{\Fc}{{\mathcal F}}
\newcommand{\Ic}{{\mathcal I}}
\newcommand{\Jc}{{\mathcal J}}
\newcommand{\Kc}{{\mathcal K}}
\newcommand{\Lc}{{\mathcal L}}
\newcommand{\Oc}{{\mathcal O}}
\newcommand{\Pc}{{\mathcal P}}
\newcommand{\Sc}{{\mathcal S}}
\newcommand{\Tc}{{\mathcal T}}
\newcommand{\Uc}{{\mathcal U}}
\newcommand{\Vc}{{\mathcal V}}

\author{Nik Weaver}

\title [A quantum Ramsey theorem]
       {A ``quantum'' Ramsey theorem for operator systems}

\address {Department of Mathematics\\
Washington University\\
Saint Louis, MO 63130}

\email {nweaver@math.wustl.edu}

\date{\em Jan.\ 13, 2016}


\begin{abstract}
Let $\mathcal{V}$ be a linear subspace of $M_n(\mathbb{C})$ which contains
the identity matrix and is stable under the formation of Hermitian adjoints.
We prove that if $n$ is sufficiently large then there exists a rank $k$
orthogonal projection $P$ such that ${\rm dim}(P\mathcal{V}P) = 1$ or $k^2$.
\end{abstract}

\maketitle

\section{Background}

An {\it operator system} in finite dimensions is a linear subspace
$\mathcal{V}$ of $M_n(\mathbb{C})$ with the properties
\begin{itemize}
\item $I_n \in \mathcal{V}$

\item $A \in \mathcal{V} \Rightarrow A^* \in \mathcal{V}$
\end{itemize}
where $I_n$ is the $n\times n$ identity matrix and $A^*$ is the Hermitian
adjoint of $A$. In this paper the scalar field
will be complex and we will write $M_n = M_n(\mathbb{C})$.

Operator systems play a role in the theory of quantum error correction.
In classical information theory, the ``confusability graph'' is a bookkeeping
device which keeps track of possible ambiguity that can result when a message
is transmitted through a noisy channel. It is defined by taking as
vertices all possible source messages, and placing an edge between two
messages if they are sufficiently similar that data corruption could lead
to them being indistinguishable on reception. Once the confusability graph
is known, one is able to overcome the problem of information loss by using an
independent subset of the confusability graph, which is known as a ``code''.
If it is agreed that only code messages will be sent, then we can be sure
that the intended message is recoverable.

When information is stored in quantum mechanical systems, the problem
of error correction changes radically. The basic theory of quantum error
correction was laid down in \cite{KL}. In \cite{DSW} it was suggested that
in this setting the role of the confusability graph is played by an operator
system, and it was shown that for every operator system a ``quantum Lov\'{a}sz
number'' could be defined, in analogy to the classical Lov\'{a}sz number of a
graph. This is an important parameter in classical information theory.
See also \cite{S} for much more along these lines.

The interpretation of operator systems as ``quantum graphs'' was also
proposed in \cite{W}, based on the more general idea of regarding linear
subspaces of $M_n$ as ``quantum relations'', and taking the conditions
$I_n \in \mathcal{V}$ and $A \in \mathcal{V} \Rightarrow A^* \in \mathcal{V}$
to respectively express reflexivity and symmetry conditions. The idea is that
the edge structure of a classical graph can be encoded in an obvious way as a
reflexive, symmetric relation on a set. This point of view was explicitly
connected to the quantum error correction literature in \cite{W2}.

Ramsey's theorem states that for any $k$ there exists $n$ such that every
graph with at least $n$ vertices contains either a $k$-clique or a
$k$-anticlique, i.e., a set of $k$ vertices among which either all edges are
present or no edges are present. Simone Severini asked the author whether
there is a ``quantum'' version of this theorem for operator systems. The
natural notions of $k$-clique and $k$-anticlique are the following.

\begin{defi}
Let $\mathcal{V} \subseteq M_n$ be an operator system. A
{\it quantum $k$-clique} of $\mathcal{V}$ is an orthogonal projection
$P \in M_n$ (i.e., a matrix satisfying $P = P^2 = P^*$) whose rank is $k$,
such that $P\mathcal{V}P = \{PAP: A \in \mathcal{V}\}$ is maximal; that is,
such that $P\mathcal{V}P = PM_nP \cong M_k$, or equivalently,
${\rm dim}(P\mathcal{V}P) = k^2$. A {\it quantum $k$-anticlique} of
$\mathcal{V}$ is a rank $k$ projection $P$ such that $P\mathcal{V}P$
is minimal; that is, such that $P\mathcal{V}P = \mathbb{C}\cdot P \cong M_1$,
or equivalently, ${\rm dim}(P\mathcal{V}P) = 1$.
\end{defi}

The definition of quantum $k$-anticlique is supported by the fact that
in quantum error correction a code is taken to be the range of a projection
satisfying just this condition, $P\mathcal{V}P = \mathbb{C}\cdot P$ \cite{KL}.
As mentioned earlier, classical codes are taken to be independent sets, which
is to say, anticliques. See also Section 4 of \cite{W2}, where intuition for
why $P\mathcal{V}P$ is correctly thought of as a ``restriction'' of
$\mathcal{V}$ is given.

The main result of this paper is a quantum Ramsey theorem which states
that for every $k$ there exists $n$ such that every operator system in
$M_n$ has either a quantum $k$-clique or a quantum $k$-anticlique. This
answers Severini's question positively. The quantum Ramsey theorem is
not merely analogous to the classical Ramsey theorem; using the bimodule
formalism of \cite{W}, we can formulate a common generalization of the
two results. This will be done in the final section of the paper.

I especially thank Michael Jury for stimulating discussions, and in
particular for conjecturing Proposition \ref{rowcolumn} and improving
Lemma \ref{normalize}.

Part of this work was done at a workshop on Zero-error information,
Operators, and Graphs at the Universitat Aut\`{o}noma de Barcelona.

\section{Examples}

If $G = (V,\mathcal{E})$ is any finite simple graph, without loss of generality
suppose $V = \{1, \ldots, n\}$ and define $\mathcal{V}_G$ to be the
operator system
$$\mathcal{V}_G = {\rm span}\{E_{ij}: i = j\mbox{ or }\{i,j\} \in \mathcal{E}\}
\subseteq M_n.$$
Here we use the notation $E_{ij}$ for the $n\times n$ matrix with a $1$ in
the $(i,j)$ entry and $0$'s elsewhere. Also, let $(e_i)$ be the standard
basis of $\mathbb{C}^n$, so that $E_{ij} = e_ie_j^*$.

The inclusion of the diagonal $E_{ii}$ matrices in $\mathcal{V}_G$ corresponds
to including a loop at each vertex in $G$. In the error correction setting
this is natural: we
place an edge between any two messages that might be indistinguishable on
reception, and this is certainly true of any message and itself. Once we
adopt the convention that every graph has a loop at each vertex, an anticlique
should no longer be a subset $S \subseteq V$ which contains no edges, it
should be a subset which contains no edges except loops. Such a set corresponds
to the projection $P_S$ onto ${\rm span}\{e_i: i \in S\}$, which has the
property
that $P_S\mathcal{V}_GP_S = {\rm span}\{E_{ii}: i \in S\}$. Or course this is
very different from a quantum anticlique where $P\mathcal{V}P$ is
one-dimensional.

To illustrate the dissimilarity between classical and quantum cliques and
anticliques, consider the {\it diagonal operator system}
$D_n \subseteq M_n$ consisting of the diagonal $n\times n$ complex
matrices. In the notation used above, this is just the operator system
$\mathcal{V}_G$ corresponding
to the empty graph on $n$ vertices. It might at first appear to falsify
the desired quantum Ramsey theorem, because of the following fact.

\begin{prop}
$D_n$ has no quantum $k$-anticlique for $k \geq 2$.
\end{prop}

\begin{proof}
Let $P \in M_n$ be a projection of rank $k \geq 2$. Since
${\rm rank}(E_{ii}) =1$ for all $i$, it follows that
${\rm rank}(PE_{ii}P) = 0$ or $1$ for each $i$. If $PE_{ii}P = 0$ for all
$i$ then $P = \sum_{i=1}^n PE_{ii}P = 0$, contradiction. Thus we must have
${\rm rank}(PE_{ii}P) = 1$ for some $i$, but then $PE_{ii}P$ cannot belong
to $\mathbb{C}\cdot P = \{aP: a \in \mathbb{C}\}$, since every matrix in
this set has rank $0$ or $k$. So $PD_nP \neq \mathbb{C}\cdot P$.
\end{proof}

Since every operator system of the form $\mathcal{V}_G$ contains the
diagonal matrices, none of these operator systems has nontrivial quantum
anticliques. The surprising thing is that for $n$ sufficiently large, they
all have quantum $k$-cliques. This follows from the next result.

\begin{prop}\label{diagonal}
If $n \geq k^2 + k -1$ then $D_n$ has a quantum $k$-clique.
\end{prop}

\begin{proof}
Without loss of generality let $n = k^2 + k - 1$. Start by considering $M_k$
acting on $\mathbb{C}^k$. Find $k^2$ vectors $v_1, \ldots, v_{k^2}$ in
$\mathbb{C}^k$ such that the rank 1 matrices $v_iv_i^*$ are linearly
independent. (For example, we could take the $k$ standard basis vectors
$e_i$ plus the $\frac{k^2 - k}{2}$ vectors $e_i + e_j$ for $i \neq j$ plus
the $\frac{k^2 - k}{2}$ vectors $e_i + ie_j$ for $i \neq j$. The corresponding
rank 1 matrices span $M_k$ and thus they must be independent since
${\rm dim}(M_k) = k^2$.) Making the identification $\mathbb{C}^n \cong
\mathbb{C}^k \oplus \mathbb{C}^{k^2 - 1}$, we can extend the $v_i$ to
orthogonal vectors $w_i \in
\mathbb{C}^n$ as follows: take $w_1 = v_1 \oplus (1, 0, \ldots, 0)$,
$w_2 = v_2 \oplus (a_1, 1, 0, \ldots, 0)$, $w_3 = v_3 \oplus (b_1, b_2, 1,
0, \ldots, 0)$, etc., with $a_1, b_1, b_2, \ldots$ successively chosen so
that $\langle w_i, w_j\rangle = 0$ for $i \neq j$. We need $k^2 - 1$ extra
dimensions to accomplish this. Now let $P$ be the rank $k$
projection of $\mathbb{C}^n$ onto $\mathbb{C}^k$ and let $D_n$ be the
diagonal operator system relative to any orthonormal basis of $\mathbb{C}^n$
that contains the vectors $\frac{w_i}{\|w_i\|}$ for $1 \leq i \leq k^2$. Then
$PD_nP$ contains $Pw_iw_i^*P = v_iv_i^*$ for all $i$, so
${\rm dim}(PD_nP) = k^2$.
\end{proof}

A stronger version of this result will be proven in Lemma \ref{blocks}.
The value $n = k^2 + k - 1$ may not be optimal, but note that in order
for $D_n$ to have a quantum $k$-clique $n$ must be at least $k^2$, since
${\rm dim}(D_n) = n$ and we need ${\rm dim}(PD_nP) = k^2$.

Next, we show that operator systems of arbitrarily large dimension may
lack quantum 3-cliques.

\begin{prop}\label{rowcolumn}
Let $\mathcal{V}_n = {\rm span}\{I_n, E_{11}, E_{12}, \ldots, E_{1n},
E_{21}, \ldots, E_{n1}\} \subseteq M_n$. Then $\mathcal{V}_n$
has no quantum 3-cliques.
\end{prop}

\begin{proof}
Let $P \in M_n$ be any projection. If $Pe_1 = 0$ then $PE_{1i}P = PE_{i1}P = 0$
for all $i$, so $P$ is a quantum
anticlique. Otherwise let $k = {\rm rank}(P)$ and let
$f_1, \ldots, f_k$ be an orthonormal basis of ${\rm ran}(P)$ with
$f_1 = \frac{Pe_1}{\|Pe_1\|}$. Then $PE_{1i}P = Pe_1e_i^*P = f_1v_i^*$ where
$v_i = \|Pe_1\|Pe_i$. The span of these matrices $f_1v_i^*$ is precisely
${\rm span}\{f_1f_i^*\}$, since the projections of the $e_i$
span ${\rm ran}(P)$. Similarly, the span of the matrices $PE_{i1}P$ is
precisely ${\rm span}\{f_if_1^*\}$. So $P\mathcal{V}_nP$ is just
$\mathcal{V}_k \subseteq M_k \cong PM_nP$, relative to the $(f_i)$ basis.
If $k \geq 3$ then ${\rm dim}(\mathcal{V}_k) = 2k < k^2$, so $P$ cannot be
a quantum clique.
\end{proof}

\section{Quantum 2-cliques}

In contrast to Proposition \ref{rowcolumn}, we will show in this section
that any operator system whose dimension is at least four must have a
quantum 2-clique. This result is clearly sharp. It is somewhat analogous
to the trivial classical fact that any graph that contains at least one
edge must have a 2-clique.

Define the Hilbert-Schmidt inner product of $A, B \in M_n$ to be
${\rm Tr}(AB^*)$. Denote the set of Hermitian $n\times n$ matrices by $M_n^h$.
Observe that any operator system is spanned by its Hermitian part since any
matrix $A$ satisfies $A = {\rm Re}(A) + i{\rm Im}(A)$ where
${\rm Re}(A) = \frac{1}{2}(A + A^*)$ and ${\rm Im}(A) = \frac{1}{2i}(A - A^*)$.

\begin{lemma}\label{rank2}
Let $\mathcal{V} \subseteq M_n$ be an operator system and suppose
${\rm dim}(\mathcal{V}) \leq 3$. Then its Hilbert-Schmidt orthocomplement
is spanned by rank 2 Hermitian matrices.
\end{lemma}

\begin{proof}
Work in $M_n^h$. Let $\mathcal{V}_0 = \mathcal{V} \cap M_n^h$ and let
$\mathcal{W}_0$ be the real span of the Hermitian matrices in
$\mathcal{V}_0^\perp$ whose rank is 2. We will show that
$\mathcal{W}_0 = \mathcal{V}_0^\perp$ (in $M_n^h$); taking complex spans
then yields the desired result.

Suppose to the contrary that there exists a nonzero Hermitian matrix
$B \in \mathcal{V}_0^\perp$ which is orthogonal to $\mathcal{W}_0$.
Say $\mathcal{V}_0 = {\rm span}\{I_n, A_1, A_2\}$, where $A_1$ and $A_2$
are not necessarily distinct from $I_n$. Since $B \in \mathcal{V}_0^\perp$,
we have ${\rm Tr}(I_nB) = {\rm Tr}(A_1B) = {\rm Tr}(A_2B) = 0$, but
${\rm Tr}(B^2) \neq 0$. We will show that there is a rank 2 Hermitian
matrix $C$ whose inner products against $I_n$, $A_1$, $A_2$, and $B$ are
the same as their inner products against $B$. This will be a matrix
in $\mathcal{W}_0$ which is not orthogonal to $B$, a contradiction.

Since $B$ is Hermitian, we can choose an orthonormal basis $(f_i)$ of
$\mathbb{C}^n$ with
respect to which it is diagonal, say $B = {\rm diag}(b_1, \ldots, b_n)$.
We may assume $b_1, \ldots, b_j \geq 0$ and $b_{j+1}, \ldots, b_n < 0$.
Let $B^+ = {\rm diag}(b_1, \ldots, b_j, 0, \ldots, 0)$ and
$B^- = {\rm diag}(0, \ldots, 0, -b_{j+1}, \ldots, -b_n)$ be the positive
and negative parts of $B$, so that $B = B^+ - B^-$. Let
$\alpha = {\rm Tr}(B^+) = {\rm Tr}(B^-)$ (they are equal since
${\rm Tr}(B) = {\rm Tr}(I_nB) = 0$).
Then $\frac{1}{\alpha}B^+$ is a convex combination of the rank 1 matrices
$f_1f_1^*$, $\ldots$, $f_jf_j^*$; that is, the linear functional
$A \mapsto \frac{1}{\alpha}{\rm Tr}(AB^+)$ is a convex combination
of the linear functionals $A \mapsto \langle Af_i, f_i\rangle$ for
$1 \leq i \leq j$. By the convexity of the joint numerical range of three
Hermitian matrices \cite{AYP}, there exists a unit vector $v \in \mathbb{C}^n$
such that $\frac{1}{\alpha}{\rm Tr}(AB^+) = \langle Av, v\rangle$ for
$A = A_1$, $A_2$, and $B$. Similarly, there exists a unit vector $w$
such that $\frac{1}{\alpha}{\rm Tr}(AB^-) = \langle Aw,w\rangle$
for $A = A_1$, $A_2$, and $B$. Then $C = \alpha(vv^* - ww^*)$ is a rank 2
Hermitian matrix whose inner products against $I_n$, $A_1$, $A_2$, and $B$
are the same as their inner products against $B$. So $C$ has
the desired properties.
\end{proof}

\begin{lemma}\label{threedim}
Let $\mathcal{V} \subseteq M_3$ be an operator system and suppose
${\rm dim}(\mathcal{V}) = 4$. Then $\mathcal{V}$ has a quantum 2-clique.
\end{lemma}

\begin{proof}
The proof is computational. Say $\mathcal{V} = {\rm span}\{A_0,A_1,A_2,A_3\}$
where $A_0 = I_3$ and the other $A_i$ are Hermitian. It will suffice to find
two vectors $v,w \in \mathbb{C}^3$ such that the four vectors
$[\langle A_iv,v\rangle, \langle A_iv,w\rangle, \langle A_iw,v\rangle,
\langle A_iw,w\rangle] \in \mathbb{C}^4$ for $0 \leq i \leq 3$ are independent.
That is, we need the $4 \times 4$ matrix whose rows are these vectors to
have nonzero determinant. Then letting $P$ be the orthogonal projection onto
${\rm span}\{v, w\}$ will verify the lemma.

We can simplify by putting the $A_i$ in a special form. First, by choosing
a basis of eigenvectors, we can assume $A_1$ is diagonal. By subtracting a
suitable multiple of $A_0$ from $A_1$, multiplying by a nonzero scalar, and
possibly reordering the basis vectors, we can arrange that $A_1$ has the
form ${\rm diag}(0, 1, a)$. (Note that ${\rm dim}(\mathcal{V}) = 4$ implies
that $A_1$ cannot be a scalar multiple of $A_0$.) These operations do not
affect ${\rm span}\{A_0,A_1,A_2,A_3\}$. Then, by subtracting suitable linear
combinations of $A_0$ and $A_1$, we can arrange that $A_2$ and $A_3$ have
the forms
$$\left[\begin{matrix}
0&a_{12}&a_{13}\cr
\bar{a}_{12}&0&a_{23}\cr
\bar{a}_{13}&\bar{a}_{23}&a_{33}
\end{matrix}\right]\qquad{\rm and}\qquad
\left[\begin{matrix}
0&b_{12}&b_{13}\cr
\bar{b}_{12}&0&b_{23}\cr
\bar{b}_{13}&\bar{b}_{23}&b_{33}
\end{matrix}\right].$$
Let
$$v = \left[\begin{matrix}
1\cr
\alpha\cr
0\end{matrix}\right]\qquad{\rm and}\qquad
w = \left[\begin{matrix}
1\cr
0\cr
\beta\end{matrix}\right],$$
then evaluate the determinant of the $4\times 4$ matrix described above
and expand it as a polynomial in $\alpha$, $\beta$, $\bar{\alpha}$, and
$\bar{\beta}$. We just need this determinant to be nonzero for some
values of $\alpha$ and $\beta$; if this fails, then the polynomial
coefficients must all be zero, and direct computation shows that this
forces one of $A_2$ and $A_3$ to be a scalar multiple of the other. We omit
the tedious but straightforward details.
\end{proof}

\begin{theo}\label{2clique}
Let $\mathcal{V} \subseteq M_n$ be an operator system and suppose
${\rm dim}(\mathcal{V}) \geq 4$. Then $\mathcal{V}$ has a quantum 2-clique.
\end{theo}

\begin{proof}
Without loss of generality we can suppose that ${\rm dim}(\mathcal{V}) = 4$.
Say $\mathcal{V} = {\rm span}\{I_n, A_1, A_2, A_3\}$ where the $A_i$ are
Hermitian and linearly independent.

We first claim that there is a projection $P$ of rank at most 3
such that $PI_nP$, $PA_1P$, and $PA_2P$ are linearly independent.
If $A_1$ and $A_2$ are jointly diagonalizable
then we can find three common eigenvectors
$v_1$, $v_2$, and $v_3$ such that the vectors $(1,1,1),
(\lambda_1, \lambda_2, \lambda_3), (\mu_1, \mu_2, \mu_3) \in \mathbb{C}^3$
are linearly independent, where $\lambda_i$ and $\mu_i$ are the
eigenvalues belonging to $v_i$ for $A_1$ and $A_2$, respectively. Then
the projection onto ${\rm span}\{v_1, v_2, v_3\}$ verifies the claim. If
$A_1$ and $A_2$ are not jointly diagonalizable, then we can find two
eigenvectors $v_1$ and $v_2$ of $A_1$ such that $\langle A_2v_1, v_2\rangle
\neq 0$. Letting $v_3$ be a third eigenvector of $A_1$ with the property
that the eigenvalues of $A_1$ belonging to $v_1$, $v_2$, and $v_3$ are not
all equal, we can again use the projection onto ${\rm span}\{v_1, v_2, v_3\}$.
This establishes the claim.

Now let $P$ be as in the claim and find $B \in M_n$ such that
$PI_nP$, $PA_1P$, $PA_2P$, and $PBP$ are linearly independent. By
Lemma \ref{threedim} we can then find a rank 2 projection $Q \leq P$
such that $QI_nQ$, $QA_1Q$, $QA_2Q$, and $QBQ$ are linearly
independent.

If $QI_nQ$, $QA_1Q$, $QA_2Q$, and $QA_3Q$ are linearly independent then
we are done. Otherwise, let $\alpha$, $\beta$, and $\gamma$ be the unique
scalars such that $QA_3Q = \alpha QI_nQ + \beta QA_1Q + \gamma QA_2Q$. By
Lemma \ref{rank2} we can find a rank 2 Hermitian matrix $C$ such that
${\rm Tr}(I_nC) = {\rm Tr}(A_1C) = {\rm Tr}(A_2C) = 0$
but ${\rm Tr}(A_3C) \neq 0$. Then $C = vv^* - ww^*$ for some
orthogonal vectors $v$ and $w$. Thus,
$\langle Av,v\rangle = \langle Aw,w\rangle$ for $A = I_n$, $A_1$, and
$A_2$, but not for $A = A_3$. It follows that the two conditions
$$\langle A_3v,v\rangle = \alpha\langle I_nv,v\rangle
+ \beta\langle A_1v,v\rangle + \gamma\langle A_2v,v\rangle$$
and
$$\langle A_3w,w\rangle = \alpha\langle I_nw,w\rangle
+ \beta\langle A_1w,w\rangle + \gamma\langle A_2w,w\rangle$$
cannot both hold. Without loss of generality suppose the first fails.
Then letting $Q'$ be the projection onto ${\rm span}({\rm ran}(Q) \cup \{v\})$,
we cannot have $Q'A_3Q' = \alpha Q'I_nQ' + \beta Q'A_1Q' + \gamma Q'A_2Q'$.
Thus ${\rm rank}(Q') = 3$ and ${\rm dim}(Q'\mathcal{V}Q') = 4$. The
theorem now follows by applying Lemma \ref{threedim} to
$Q'\mathcal{V}Q'$.
\end{proof}

Theorem \ref{2clique} does not generalize to arbitrary four-dimensional
subspaces of $M_n$. For instance, let
$\mathcal{V} = {\rm span}\{E_{11}, E_{12}, E_{13}, E_{14}\} \subset M_4$;
by reasoning similar to that in the proof of Proposition \ref{rowcolumn},
if $P$ is any rank $2$ projection in $M_4$ then ${\rm dim}(P\mathcal{V}P)
\leq 2$.

\section{The main theorem}

The proof of our main theorem proceeds through a series of lemmas.

\begin{lemma}\label{basic}
Suppose the operator system $\mathcal{V}$ is contained in $D_n$. If
${\rm dim}(\mathcal{V}) \geq k^2 + k - 1$ then $\mathcal{V}$ has a
quantum $k$-clique. If ${\rm dim}(\mathcal{V}) \leq \frac{n-k}{k-1}$
then $\mathcal{V}$ has a quantum $k$-anticlique. If $n \geq k^3 - k + 1$
then $\mathcal{V}$ has either a quantum $k$-clique or a quantum
$k$-anticlique.
\end{lemma}

\begin{proof}
If ${\rm dim}(\mathcal{V}) \geq k^2 + k - 1 = m$ then we can find a
set of indices $S \subseteq \{1, \ldots, n\}$ of cardinality $m$ such that
${\rm dim}(P\mathcal{V}P) = m$ where $P$ is the orthogonal projection onto
${\rm span}\{e_i: i \in S\}$. Then $P\mathcal{V}P \cong D_m \subseteq M_m
\cong PM_nP$ and Proposition \ref{diagonal} yields that $P\mathcal{V}P$,
and hence also $\mathcal{V}$, has a quantum $k$-clique. If
${\rm dim}(\mathcal{V}) \leq \frac{n-k}{k-1}$ then a result of Tverberg
\cite{T1, T2} can be used to extract a quantum $k$-anticlique; this is
essentially Theorem 4 of \cite{KLV}. Thus if $k^2 + k - 1 \leq
\frac{n-k}{k-1}$ then one of the two cases must obtain, i.e., $\mathcal{V}$
must have either a quantum $k$-clique or a quantum $k$-anticlique. A little
algebra shows that this inequality is equivalent to $n \geq k^3 - k + 1$.
\end{proof}

\begin{lemma}\label{normalize}
Let $v_1, \ldots, v_r$ be vectors in
$\mathbb{C}^s$. Then there are vectors $w_1, \ldots, w_r \in \mathbb{C}^{r-1}$
such that the vectors $v_i \oplus w_i \in \mathbb{C}^{s + r - 1}$ are
pairwise orthogonal and all have the same norm.
\end{lemma}

\begin{proof}
Let $G$ be the Gramian matrix of the vectors $v_i$ and let $\|G\|$ be
its operator norm. Then ${\rm rank}(\|G\|I_r - G) \leq r - 1$,
so we can find vectors $w_i \in \mathbb{C}^{r-1}$ whose Gramian matrix
is $\|G\|I_r - G$. The Gramian matrix of the vectors $v_i \oplus w_i$
is then $\|G\|I_r$, as desired.
\end{proof}

Then next lemma improves Proposition \ref{diagonal}.

\begin{lemma}\label{blocks}
Let $n = k^2 + k - 1$ and suppose
$A_1, \ldots, A_{k^2}$ are Hermitian matrices in $M_n$ such that
for each $i$ we have $\langle A_ie_i, e_i\rangle = 1$, and also
$\langle A_ie_r,e_s\rangle = 0$ whenever ${\rm max}\{r,s\} > i$.
Then $\mathcal{V} = {\rm span}\{I, A_1, \ldots, A_{k^2}\}$ has a quantum
$k$-clique.
\end{lemma}

\begin{proof}
Let $A_i$ have matrix entries $(a^i_{rs})$. The
goal is to find vectors $v_1, \ldots, v_{k^2} \in \mathbb{C}^k$ such that
the matrices
$$A_i' = \sum_{1 \leq r,s \leq k^2} a^i_{rs}v_rv_s^* \in M_k$$
are linearly independent. Once we have done this, find vectors
$w_i \in {\bf C}^{k^2 - 1}$ as in Lemma \ref{normalize} and let
$f_i = \frac{1}{N}(v_i \oplus w_i) \in \mathbb{C}^n \cong \mathbb{C}^k
\oplus \mathbb{C}^{k^2 - 1}$ where $N$ is the
common norm of the $v_i \oplus w_i$. Then the $f_i$ form an orthonormal
set in $\mathbb{C}^n$, so they can be extended to an orthonormal basis,
and the operators whose matrices for this basis are the $A_i$ compress to
the matrices $\frac{1}{N^2}A_i'$ on the initial $\mathbb{C}^k$, which
are linearly independent. So $P\mathcal{V}P$ contains $k^2$ linearly
independent matrices, where $P$ is the orthogonal projection onto
$\mathbb{C}^k$, showing that $\mathcal{V}$ has a quantum $k$-clique.
\medskip

The vectors $v_i$ are constructed inductively. Once $v_1, \ldots, v_i$ are
chosen so that $A_1', \ldots, A_i'$ are independent, future choices
of the $v$'s cannot change this since $A_1, \ldots, A_i$ all live on the
initial $i\times i$ block. We can let $v_1$ be any nonzero vector in
$\mathbb{C}^k$, since $A_1 = e_1e_1^*$, so that $A_1' = v_1v_1^*$ and this
only has to be nonzero. Now suppose $v_1, \ldots, v_{i-1}$ have been chosen
and we need to select $v_i$ so that $A_i'$ is independent of
$A_1', \ldots, A_{i-1}'$. After choosing $v_i$ we will have
$A_i' = \sum_{1 \leq r,s \leq i} a^i_{rs}v_rv_s^*$.
Let $B$ be this sum restricted to $1 \leq r,s \leq i-1$. That part is
already determined since $v_i$ does not appear. Also let
$$u = a^i_{1i}v_1 + \cdots + a^i_{(i-1)i}v_{i-1};$$
then we will have
$$A_i' = B + uv_i^* + v_iu^* + v_iv_i^*$$
(using the assumption that $a^i_{ii} = 1$). That is,
$$A_i' = (B - uu^*) + (u + v_i)(u + v_i)^* = B' + \tilde{u}\tilde{u}^*$$
where $\tilde{u} = u + v_i$ is arbitrary,
and the question is whether $\tilde{u}$ can be chosen to make this matrix
independent of $A_1', \ldots, A_{i-1}'$. But the possible choices of
$A_i'$ span $M_k$ --- there is no matrix which is Hilbert-Schmidt
orthogonal to $B' + \tilde{u}\tilde{u}^*$ for all $\tilde{u}$ --- so there
must be a choice of $\tilde{u}$ which makes $A_i'$ independent of
$A_1', \ldots, A_{i-1}'$, as desired.
\end{proof}

Next we prove a technical variation on Lemma \ref{blocks}.

\begin{lemma}\label{blocks2}
Let $n = k^4 + k^3 + k - 1$ and let $\mathcal{V}$ be an operator system
contained in $M_n$. Suppose $\mathcal{V}$ contains matrices
$A_1, \ldots, A_{k^4 + k^3}$ such that for each $i$ we have
$\langle A_ie_i, e_{i+1}\rangle \neq 0$, and also
$\langle A_ie_r, e_s\rangle = 0$ whenever ${\rm max}\{r,s\} > i+1$ and
$r \neq s$. Then $\mathcal{V}$ has a quantum $k$-clique.
\end{lemma}

\begin{proof}
Let $A_i$ have matrix entries $(a^i_{rs})$.
Observe that for each $i$ the compression of $A_i$ to
${\rm span}\{e_{i+2}, \ldots, e_n\}$ is diagonal. For each $r > i+1$
let the {\it $r$-tail} of $A_i$ be the vector $(a^i_{rr}, \ldots, a^i_{nn})$.
Suppose there exist indices $i_1, \ldots, i_{k^2 + k - 1}$ such that
the $r$-tails of the $A_{i_j}$, $1 \leq j \leq k^2 + k - 1$, are
linearly independent, where $r = {\rm max}_j\{i_j + 2\}$. Then the compression
of $\mathcal{V}$ to ${\rm span}\{e_r, \ldots, e_n\}$ contains $k^2 + k - 1$
linearly independent diagonal matrices, so it has a quantum $k$-clique by
the first assertion of Lemma \ref{basic}. Thus, we may assume that for
any $k^2 + k - 1$ distinct indices $i_j$ the matrices $A_{i_j}$ have
linearly dependent $r$-tails.

We construct an orthonormal sequence of vectors $v_i$ and a sequence of
Hermitian matrices $B_i \in \mathcal{V}$, $1 \leq i \leq k^2$, such that
the compressions of the $B_i$ to ${\rm span}\{v_1, \ldots, v_{k^2},
e_{k^4 + k^3 + 1}, \ldots, e_{k^4 + k^3 + k - 1}\}$ satisfy the hypotheses
of Lemma \ref{blocks}. This will ensure the existence of a quantum
$k$-clique.

The first $k^2 + k - 1$ matrices $A_1, \ldots, A_{k^2 + k - 1}$ have
linearly dependent $r$-tails for $r = k^2 + k + 1$. Thus there is a
nontrivial linear combination
$B'_1 = \sum_{i=1}^{k^2 + k - 1} \alpha_i A_i$ whose $r$-tail is the
zero vector. Letting $j$ be the largest index such that $\alpha_j$ is
nonzero, we have $\langle B'_1e_j, e_{j+1}\rangle \neq 0$ because
$\langle A_je_j,e_{j+1}\rangle \neq 0$ but
$\langle A_ie_j,e_{j+1}\rangle = 0$ for $i < j$. Thus the compression
of $B'_1$ to ${\rm span}\{e_1, \ldots, e_{k^2 + k}\}$ is nonzero, so
there exists a unit vector $v_1$ in this span such that
$\langle B'_1v_1, v_1\rangle \neq 0$. Then let $B_1$ be a scalar
multiple of either the real or imaginary part of $B_1'$ which satisfies
$\langle B_1v_1,v_1\rangle = 1$. Note that $\langle B_1e_r,e_s\rangle = 0$
for any $r,s$ with ${\rm max}\{r,s\} > k^2 + k$. Apply the same reasoning
to the next block of $k^2 + k - 1$ matrices
$A_{k^2 + k + 1}, \ldots, A_{2k^2 + 2k - 1}$ to find $v_2$ and $B_2$,
and proceed inductively. After $k^2$ steps, $k^2(k^2 + k) = k^4 + k^3$
indices will have been used up and $k - 1$ (namely, $e_{k^4 + k^3 + 1},
\ldots, e_{k^4 + k^3 + k-1}$) will remain, as needed.
\end{proof}

\begin{theo}\label{main}
Every operator system in $M_{8k^{11}}$
has either a quantum $k$-clique or a quantum $k$-anticlique.
\end{theo}

\begin{proof}
Set $n = 8k^{11}$ and let $\mathcal{V}$ be an operator system in $M_n$.
Find a unit vector $v_1 \in \mathbb{C}^n$, if one exists, such that
the dimension of $\mathcal{V}v_1 = \{Av_1: A \in \mathcal{V}\}$
is less than $8k^8$. Then find a unit vector $v_2 \in (\mathcal{V}v_1)^\perp$,
if one exists, such that the dimension of
$(\mathcal{V}v_1)^\perp \cap (\mathcal{V}v_2)$ is less than $8k^8$.
Proceed in this fashion, at the $r$th step trying to find a unit vector
$v_r$ in
$$(\mathcal{V}v_1)^\perp \cap \cdots \cap (\mathcal{V}v_{r-1})^\perp$$
such that the dimension of
$$(\mathcal{V}v_1)^\perp \cap \cdots \cap (\mathcal{V}v_{r-1})^\perp
\cap (\mathcal{V}v_r)$$
is less than $8k^8$. If this construction lasts for $k^3$
steps then the compression of $\mathcal{V}$ to
${\rm span}\{v_1, \ldots, v_{k^3}\} \cong M_{k^3}$ is contained in $D_{k^3}$,
so this compression, and hence also $\mathcal{V}$, has either a quantum
$k$-clique or a quantum $k$-anticlique by Lemma \ref{basic}.

Otherwise, the construction fails at some stage $d$. This means that
the compression $\mathcal{V}'$ of $\mathcal{V}$ to
$F = (\mathcal{V}v_1)^\perp \cap \cdots \cap (\mathcal{V}v_d)^\perp$ has
the property that the dimension of $\mathcal{V}'v$ is at least $8k^8$,
for every unit vector $v \in F$.

Work in $F$.
Choose any nonzero vector $w_1 \in F$ and find $A_1 \in \mathcal{V}'$ such
that $w_2 = A_1w_1$ is nonzero and orthogonal to $w_1$. Then find
$A_2 \in \mathcal{V}'$ such that $w_3 = A_2w_2$ is nonzero and orthogonal
to ${\rm span}\{w_1, w_2, A_1w_1, A_1^*w_1, A_1w_2, A_1^*w_2\}$. Continue
in this way, at the $r$th step finding $A_r \in \mathcal{V}'$ such that
$w_{r+1} = A_rw_r$ is nonzero and orthogonal to
${\rm span}\{w_j, A_iw_j, A_i^*w_j: i < r$ and $j \leq r\}$. The dimension of
this span is at most $2r^2 - r$, so as long as $r \leq 2k^4$ its dimension
is less than $8k^8$ and a vector $w_{r+1}$ can be found. Compressing to
the span of the $w_i$ then puts us in the situation of Lemma \ref{blocks2}
with $n = 2k^4$, which is more than enough. So there exists a quantum
$k$-clique by that lemma.
\end{proof}

The constants in the proof could easily be improved, but only marginally. Very
likely the problem of determining optimal bounds on quantum Ramsey numbers is
open-ended, just as in the classical case.

\section{A generalization}

In this section we will present a result which simultaneously generalizes the
classical and quantum Ramsey theorems. This is less interesting than it sounds
because the proof involves little more than a reduction to these two special
cases. Perhaps the statement of the theorem is more significant than its proof.

At the beginning of Section 2 we showed how any simple graph $G$ on the vertex
set $\{1, \ldots, n\}$ gives rise to an operator system $\mathcal{V}_G
\subseteq M_n$. This operator system has the special property that it is a
bimodule over $D_n$, i.e., it is stable under left and right multiplication
by diagonal matrices. Conversely, it is not hard to see that any operator
system in $M_n$ which is also a $D_n$-$D_n$-bimodule must have the form
$\mathcal{V}_G$ for some $G$ (\cite{W}, Propositions 2.2 and 2.5). The
general definition therefore goes as follows:

\begin{defi}\label{qgdef}
(\cite{W}, Definition 2.6 (d))
Let $\mathcal{M}$ be a unital $*$-subalgebra of $M_n$. A {\it quantum graph
on $\mathcal{M}$} is an operator system $\mathcal{V} \subseteq M_n$ which
satisfies $\mathcal{M}'\mathcal{V}\mathcal{M}' = \mathcal{V}$.
\end{defi}

Here $\mathcal{M}' = \{A \in M_n: AB = BA$ for all $B \in \mathcal{M}\}$ is
the commutant of $\mathcal{M}$. This definition is actually
representation-independent: if $\mathcal{M}$ and $\mathcal{N}$ are
$*$-isomorphic unital $*$-subalgebras of two matrix algebras (possibly
of different sizes), then the quantum graphs on $\mathcal{M}$ naturally
correspond to the quantum graphs on $\mathcal{N}$ (\cite{W}, Theorem 2.7).
More properly, one could say that the pair $(\mathcal{M}, \mathcal{V})$
is the quantum graph, just as a classical graph is a pair
$(V, \mathcal{E})$.

If $\mathcal{M} = M_n$ then its commutant
is $\mathbb{C}\cdot I_n$ and the bimodule condition in Definition
\ref{qgdef} is vacuous: any
operator system in $M_n$ is a quantum graph on $M_n$. On the other hand,
the commutant of $\mathcal{M} = D_n$ is itself, so that by the comment
made above, the quantum graphs
on $D_n$ --- the operator systems which are $D_n$-$D_n$-bimodules ---
correspond to simple graphs on the vertex set $\{1, \ldots, n\}$.
In this correspondence,
subsets of $\{1, \ldots, n\}$ give rise to orthogonal projections
$P \in D_n$, and the $k$-cliques and $k$-anticliques of the graph are
realized in the matrix picture as rank $k$ orthogonal projections $P \in D_n$
which satisfy $P\mathcal{V}_GP = PM_nP$ or $PD_nP$, respectively. This
suggests the following definition.

\begin{defi}
Let $\mathcal{M}$ be a unital $*$-subalgebra of $M_n$ and let
$\mathcal{V} \subseteq M_n$ be a quantum graph on $\mathcal{M}$. A
rank $k$ projection $P \in \mathcal{M}$ is a {\it quantum $k$-clique} if
it satisfies $P\mathcal{V}P = PM_nP$ and a {\it quantum $k$-anticlique} if
it satisfies $P\mathcal{V}P = P\mathcal{M}'P$.
\end{defi}

Since every operator system contains the identity matrix, if $\mathcal{V}$
is a quantum graph on $\mathcal{M}$ then $\mathcal{M}' \subseteq \mathcal{V}$.
So $P\mathcal{V}P = P\mathcal{M}'P$ is the minimal possibility, as
$P\mathcal{V}P = PM_nP$ is the maximal possibility. Note the crucial
requirement that $P$ must belong to $\mathcal{M}$.

If $\mathcal{M} = M_n$ then $\mathcal{M}' = \mathbb{C}\cdot I_n$ and
the preceding definition duplicates the notions of quantum $k$-clique and
quantum $k$-anticlique used earlier in the paper, whereas if
$\mathcal{M} = D_n$ it effectively reproduces the classical notions of
$k$-clique and $k$-anticlique in a finite simple graph. In the classical
setting fewer operator systems count as graphs, but one also has less
freedom in the choice of $P$ when seeking cliques or anticliques.

We require only the following simple lemma.

\begin{lemma}
Let $\mathcal{V} \subseteq M_{nd} \cong M_n \otimes M_d$ be a quantum
graph on $M_n\otimes I_d$. If $nd \geq 8k^{11}$ then there is a projection
in $M_n \otimes I_d$ whose rank is at least $k$, and which is either a
quantum clique or a quantum anticlique of $\mathcal{V}$.
\end{lemma}

\begin{proof}
Since $\mathcal{V}$ is a bimodule over $(M_n\otimes I_d)' = I_n \otimes M_d$,
it has the form $\mathcal{V} = \mathcal{W}\otimes M_d$ for some operator
system $\mathcal{W} \subseteq M_n$. If $d = 1$ then the desired result was
proven in Theorem \ref{main}, and if $d \geq k$ then any projection of the
form $P\otimes M_d$, where $P$ is a rank 1 projection in $M_n$, will have
rank at least $k$ and be both a quantum clique
and a quantum anticlique. So assume $2 \leq d < k$.

Now if $d \geq 3$ then $d^{10/11} > 2$, so $d < k(d^{10/11} - 1)$. Thus
$1 < \frac{k}{d}(d^{10/11} - 1)$, i.e., $\frac{k}{d} + 1 < \frac{k}{d}\cdot
d^{10/11} = \frac{k}{d^{1/11}}$, which implies $(\frac{k}{d} + 1)^{11}
< \frac{k^{11}}{d}$. So finally
$$n \geq \frac{8k^{11}}{d} > 8\left(\frac{k}{d} + 1\right)^{11}
> 8\left\lceil\frac{k}{d}\right\rceil^{11}.$$
If $d = 2$ then $d < 3(d^{10/11} - 1)$, so the same reasoning leads
to the same inequality $n \geq 8\lceil\frac{k}{d}\rceil^{11}$ provided
$k \geq 3$, and the inequality is immediate
when $k = d = 2$. So we conclude that in all cases $n \geq
8\lceil\frac{k}{d}\rceil^{11}$. By Theorem \ref{main}, $\mathcal{W}$ has
a quantum $\lceil\frac{k}{d}\rceil$-clique or a quantum 
$\lceil\frac{k}{d}\rceil$-anticlique $Q \in M_n$. Then $Q\otimes I_d$ is
correspondingly either a quantum $\lceil\frac{k}{d}\rceil\cdot d$-clique or
a quantum  $\lceil\frac{k}{d}\rceil\cdot d$-anticlique of $\mathcal{V}$,
which is enough.
\end{proof}

Note that we cannot promise a quantum $k$-clique or -anticlique, only
a $\geq k$-clique or -anticlique, since the rank of any projection in
$M_n\otimes I_d$ is a multiple of $d$.

\begin{theo}\label{general}
For every $k$ there exists $n$ such that if $\mathcal{M}$ is a unital
$*$-subalgebra of $M_n$ and $\mathcal{V} \subseteq M_n$ is an operator
system satisfying $\mathcal{M}'\mathcal{V}\mathcal{M}' = \mathcal{V}$,
then there is a projection $P \in \mathcal{M}$ whose rank is at least
$k$ and such that $P\mathcal{V}P = PM_nP$ or $P\mathcal{M}'P$.
\end{theo}

\begin{proof}
Let $R(k,k)$ be the classical Ramsey number and set $n = 8k^{11}\cdot R(k,k)$.
Now $\mathcal{M}$ has the form $(M_{n_1}\otimes I_{d_1})\oplus
\cdots \oplus (M_{n_r}\otimes I_{d_r})$ for some pair of sequences
$(n_1, \ldots, n_r)$ and $(d_1, \ldots, d_r)$ such that $n_1d_1 + \cdots
+ n_rd_r = n$. Thus if $r \leq R(k,k)$ then for some $i$ we must have
$n_id_i \geq 8k^{11}$, and compressing to that block then yields the
desired conclusion by appealing to the lemma. Otherwise, if $r > R(k,k)$,
then choose a sequence of rank 1 projections $Q_i \in M_{n_i}$ and
work in $QM_nQ$ where $Q = (Q_1 \otimes I_{d_1}) \oplus \cdots \oplus
(Q_r\otimes I_{d_r})$. Then $QM_nQ \cong M_{d_1 + \cdots + d_r}$,
$Q\mathcal{M}Q \cong \mathbb{C}\cdot I_{d_1} \oplus \cdots \oplus
\mathbb{C}\cdot I_{d_r} \cong D_r$, and $Q\mathcal{V}Q$
is a bimodule over the commutant of $Q\mathcal{M}Q$ in $QM_nQ$, i.e., the
$*$-algebra $M_{d_1} \oplus \cdots \oplus M_{d_r}$. It follows that there
is a graph $G = (V,\mathcal{E})$ on the vertex set $V = \{1, \ldots, r\}$ such that
$Q\mathcal{V}Q$ has the form
$$Q\mathcal{V}Q = \sum E_{ij}\otimes M_{d_id_j},$$
taking the sum over the set of pairs $\{(i,j): i = j$ or
$\{i,j\} \in \mathcal{E}\}$ (\cite{W}, Theorem 2.7).
Since $r > R(k,k)$, there exists either a $k$-clique or a $k$-anticlique
in $G$, and this gives rise to a diagonal projection in $Q\mathcal{M}Q$
whose rank is at least $k$ and which is either a quantum $k$-clique or a
quantum $k$-anticlique of $\mathcal{V}$.
\end{proof}

Again, when $\mathcal{M} = M_n$ Theorem \ref{general} recovers the quantum
Ramsey theorem and when $\mathcal{M} = D_n$ it recovers the classical
Ramsey theorem (though in both cases with worse constants).

Theorem \ref{general} could also be proven by mimicking the proof of
Theorem \ref{main}. However, in order to accomodate the requirement that
$P$ belong to $\mathcal{M}$ we need to modify the last part of the proof
so as to be sure that each $w_r$ belongs to $\mathcal{W}w_1 \cap \cdots
\cap \mathcal{W}w_{r-1}$. This means that instead of needing
$\mathcal{W}v$ to have sufficiently large dimension for each $v$, we
need it to have sufficiently small codimension. Ensuring that this
must be the case if the construction in the first part of the proof fails
then requires that construction to take place in a space whose dimension
is exponential in $k$. This explains the dramatic difference between
classical and quantum Ramsey numbers (the first grows exponentially, the
second polynomially).



\begin{thebibliography}{aaaaaaaa}

\bibitem{AYP}
Y.\ H.\ Au-Yeung and Y.\ T.\ Poon, A remark on the convexity and
positive definiteness concerning Hermitian matrices, {\it Southeast
Asian Bull.\ Math.\ \bf 3} (1979), 85-92.

\bibitem{DSW}
R.\ Duan, S.\ Severini, and A.\ Winter, Zero-error communication via
quantum channels, noncommutative graphs, and a quantum Lov\'{a}sz number,
{\it IEEE Trans.\ Inform.\ Theory \bf 59} (2013), 1164-1174. 

\bibitem{KL}
E.\ Knill and R.\ Laflamme, Theory of quantum error-correcting codes,
{\it Phys.\ Rev.\ A \bf 55} (1997), 900-911.

\bibitem{KLV}
E.\ Knill, R.\ Laflamme, and L.\ Viola, Theory of quantum error
correction for general noise, {\it Phys.\ Rev.\ Lett.\ \bf 84} (2000),
2525-2528.

\bibitem{S}
D.\ Stahlke, Quantum source-channel coding and non-commutative graph
theory, arXiv:1405.5254.

\bibitem{T1}
H.\ Tverberg, A generalization of Radon's theorem, {\it J.\ London
Math.\ Soc.\ \bf 41} (1966), 123-128.

\bibitem{T2}
{---------}, A generalization of Radon's theorem, II, {\it Bull.\
Austral.\ Math.\ Soc.\ \bf 24} (1981), 321-325.

\bibitem{W}
N.\ Weaver, Quantum relations, {\it Mem.\ Amer.\ Math.\ Soc.\ \bf 215}
(2012), v-vi, 81-140.

\bibitem{W2}
{---------}, Quantum graphs as quantum relations, arXiv:1506.03892.

\end{thebibliography}
\end{document}